\documentclass[11pt]{amsart}[12]
\usepackage{mathrsfs, amsmath,amsfonts,amssymb}
\usepackage{amsmath}
\usepackage{mathtools}
\usepackage{fullpage}
\newtheorem{theorem}{Theorem}
\newtheorem{proposition}{Proposition}\newtheorem{corollary}{Corollary}
\newtheorem{lemma}{Lemma}
\newtheorem{rem}{Remark}
\title{Explicit expressions of the Hua-Pickrell Semi-group} 
\author[J. Arista]{Jonas Arista}
\address{School of Mathematics and Statistics \\ University College Dublin \\ Dublin 4 \\ Ireland}
\email{jonas.aristacarrera@ucd.ie}
\author[N. Demni]{Nizar Demni}
\address{IRMAR, Universit\'e de Rennes 1\\ Campus de
Beaulieu\\ 35042 Rennes cedex\\ France}
\email{nizar.demni@univ-rennes1.fr}
\keywords{Hua-Pickrell diffusion; Routh-Romanovski polynomials; Associated Legendre function; Exponential functionals; Bougerol's identity.}
\begin{document}

\begin{abstract}
In this paper, we study the one-dimensional Hua-Pickrell diffusion. We start by revisiting the stationary case considered by E. Wong for which we supply omitted details and write down a unified expression of its semi-group density through the associated Legendre function in the cut. Next, we focus on the general (not necessarily stationary) case for which we prove an intertwining relation between Hua-Pickrell diffusions corresponding to different sets of parameters. Using Cauchy Beta integral on the one hand and Girsanov's Theorem on the other hand, we discuss the connection between the stationary and general cases. Afterwards, we prove our main result providing novel integral representations of the Hua-Pickrell semi-group density, answering a question raised by Alili, Matsumoto and Shiraishi (S\'eminaire de Probabilit\'es, 35, 2001). To this end, we appeal to the semi-group density of the Maass Laplacian and extend it to purely-imaginary values of the magnetic field. 
In the last section, we use the Karlin-McGregor formula to derive an expression of the semi-group density of the multi-dimensional Hua-Pickrell particle system introduced by T. Assiotis. 
 \end{abstract} 
 
 \maketitle

\section{Introduction}
The Hua-Pickrell (HP) process is a diffusion valued in the whole real line whose generator acts on smooth functions as: 
\begin{equation}\label{Generator}
\mathscr{L}_{A,K} := (1+u^2) \frac{d^2}{du^2}+ (2Au+K) \frac{d}{du}, \quad A, K \in \mathbb{R}. 
\end{equation}
The name `Hua-Pickrell' is borrowed from \cite{Assio} and may be misleading since the operator \eqref{Generator} appeared long-time ago in relation to Routh-Romanovski (or pseudo-Jacobi) polynomials. Nonetheless, it may be justified by the fact that the matrix-variate analogue of this diffusion introduced in \cite{Assio} is ergodic with stationary distribution given by the so-called HP measure on the space of Hermitian matrices (see also \cite{Buf-Qiu},\cite{BNR}). Up to our best knowledge, the special instance $A < 1/2, K=0,$ of the HP diffusion appeared for the first time in E. Wong's work \cite{Wong} (see Example (E)). There, the author classified stationary diffusions whose generators are Sturm-Liouville operators and whose stationary distributions solve the Pearson equation. The restriction to these parameters not only ensures the stationarity of the HP diffusion but also its invariance in distribution under the flip $u \rightarrow -u$ when starting at zero. More generally, the long-time behavior of the HP diffusion may be read off from the connection it has with exponential functionals of Brownian motions with drifts. More precisely, the variable change $y = \sinh(u)$ transforms \eqref{Generator} into the following generator:
\begin{equation}\label{Generator1}
\frac{d^2}{d^2y}+ \left[(2A-1) \tanh(y)+\frac{K}{\cosh(y)}\right] \frac{d}{dy}, 
\end{equation}
which was studied in \cite{ADY} and \cite{AMS} in the context of a generalisation of the classical Bougerol's identity. Besides, if $\left(Y^{(A,K)}_{t}, t\geq 0\right)$ is the diffusion corresponding to \eqref{Generator1} starting at $Y^{(A,K)}_0\in\mathbb{R}$, then the HP process $\sinh(Y^{(A,K)}_{t})$, $t\geq0$, has the same distribution as (\cite{ADY}, Proposition 1):
 \begin{equation*}
e^{B_{2t}^{(A-1/2)}} \left(\sinh(Y^{(A,K)}_0)+ \int_0^{2t} e^{-B_s^{(A-1/2)}} d\gamma_s^{(K/2)}\right), \quad t \geq 0,
\end{equation*}
where $B^{(A-1/2)}$ and $\gamma^{(K/2)}$ are independent standard Brownian motions with drifts $A-1/2$ and $K/2$, respectively. In particular, this connection answers a question raised in \cite{Wong} (see Section 3 there). Moreover, the time-reversal property of the Brownian motion shows that for fixed $t>0$, $\sinh(Y^{(A,K)}_{t})$ is distributed as
\begin{equation}\label{EQD} 
\sinh(Y^{(A,K)}_0) e^{B_{2t}^{(A-1/2)}}+ \int_0^{2t} e^{B_s^{(A-1/2)}} d\gamma_s^{(K/2)}.
\end{equation}
Consequently, the Dubins-Scharwtz Theorem implies that the HP process converges in distribution as $t \rightarrow +\infty$ if $A < 1/2$ and oscillates otherwise.

At the geometrical side, the operator 
\begin{equation}\label{G2}
\frac{d^2}{d^2y}+ (N+1) \tanh(y), 
\end{equation}
where $N \geq 0$ is a positive integer, is the radial part of the Laplace-Beltrami operator of the one-sheeted hyperboloid (de Sitter space):
\begin{equation*}
\{x_1^2+\dots + x_{N+2}^2 - x_{N+3}^2 = 1\} \approx \, SO_0(N+2,1)/SO_0(N+1,1).
\end{equation*}
Here, $SO_0(N+1,1)$ is the identity component of the pseudo-orthogonal group $SO(N,1)$ (see \cite{Hec-Sch}, part II, and references therein). In particular, if $N=1$ then the corresponding de Sitter space is realised as:
\begin{equation*}
SO(3,1)/SO(2,1) \approx Sl(2,\mathbb{C})/Sl(2,\mathbb{R}), 
\end{equation*}
and the radial part of its Laplace-Beltrami operator is the generator of the process $\left(B_t+\epsilon t, t\geq0\right)$, where $\epsilon$ is a symmetric Bernoulli random variable which is independent of $B$ (\cite{ADY}, see also Example 2 in \cite{Pit-Rog}). Another striking and interesting observation made in \cite{ADY} reveals that the operator \eqref{G2} coincides also with the Laplace-Beltrami operator in the real hyperbolic space 
$H^{N+2} \subset \mathbb{R}^{N+3}$ written in equidistant coordinates. In the mathematical physics realm, the Sturm-Liouville operator $\mathscr{L}_{A,K}$ may be mapped into the Hamiltonian with hyperbolic scarf potential: 
\begin{equation}\label{Scarf}
-\frac{d^2}{du^2} + C_1(A, K) \frac{\sinh(u)}{\cosh^2(u)} + C_2(A, K) \frac{1}{\cosh^2(u)}, \quad u \in \mathbb{R},
\end{equation}
where $C_1(A, K), C_2(A, K),$ are constants depending on $(A,K)$ (see e.g. \cite{Pee}, \cite{RWAK}). 

In representation-theoretical words, this Hamiltonian admits a $SU(1,1)$-symmetry (\cite{DKS}, \cite{KS}). 

Despite the multiple occurrences of the HP generator $\mathscr{L}_{A,K}$ in its various forms and though it is a Sturm-Liouville operator, still many rich clarifications related to its semi-group density are significant being highlighted. To start with, there is a typo in Wong's expression of the eigenfunctions corresponding to the continuous spectrum and the present work develops the details leading to the correct one. Next, we prove that the expansion of the semi-group density in the stationary case given in \cite{Wong} may be written in a unified way using the associated Legendre function in the cut. The newly-obtained expression has the merit to extend to the HP particle system after the use of the Karlin-McGregor and Andreief formulas. The main ingredient used to prove the unified expression is a connection    
we prove between `symmetric' Romanovski polynomials and associated Legendre function in the cut. This connection provides also the details of the derivation of another formula given in \cite{Wong} for the semi-group density in the case when $A$ is half of integer. For general parameters $(A,K)$ (not necessarily stationary case), we exhibit an intertwining property between stationary and non stationary HP processes with parameters $(A,K)$ and $(2-A,-K)$ respectively. This property extends the one given in \cite{Mat-Yor} when $K=0$. Using Cauchy Beta integral on the one hand and some absolute continuity relations given by Girsanov's Theorem on the other hand, we discuss the connection between the stationary and general cases. Afterwards, we prove our main result and derive integral representations of the HP semi-group density valid for any $A \neq 1, K \in \mathbb{R}$, avoiding the Sturm-Liouville theory and rather inverting the Fourier transform derived in \cite{AMS}, Proposition 6.1. To this end, we use the analytic extension of the semi-group density of the Maass Laplacian to purely-imaginary values of the magnetic field and appeal to the connection between the Maass and the Morse semi-groups (see \cite{Ike-Mat}). In particular, when $A \neq 1, K=0$, we get a new expression of the HP semi-group density as a partial Mellin transform of the semi-group density of the Brownian motion in the hyperbolic plane. Finally, we give some interest in the multivariate analogue of the HP process introduced and studied by T. Assiotis (\cite{Assio}). This is a one-dimensional particle system depending on a complex parameter and realised as a Doob transform of independent stationary HP processes. Using the Karlin-McGregor and Andreief formulas, we express the semi-group density of this particle system with real parameter through multivariate symmetric Legendre functions. 

\subsection{Organisation of the paper}
In Section 2 we revisit the stationary case of Wong's paper \cite{Wong}. There, we correct the eigenfunction expression corresponding to the continuous spectrum and prove the connection between `symmetric' Romanovski polynomials and associated Legendre function in the cut. We close that section by describing the main steps toward Wong's integral representation of the HP semi-group density proved for half-integer values of $A$. In Section 3, we consider the HP diffusion with general parameters for which we first recall the spectral decomposition of its generator then prove the intertwining property described above. We also prove in the same section an identity transforming a `symmetric' Romanovski polynomial to a `non-symmetric' one which may be seen as another type of intertwining since it involves the Cauchy Beta kernel. These intertwinings are next related to Girsanov's Theorem connecting the stationary and general cases. In Section 4, we prove our main result and derive an integral representation of the HP semi-group density. In the last section, we recall basic facts on the HP particle system and derive an expression for its semi-group density when the parameter it depends on is real.

\subsection{Acknowledgements}
The authors appreciated stimulating discussions with R. Chhaibi, F. Baudoin, L. Miclo, H. Matsumoto and N. O'Connell. The first author gratefully acknowledges the support of the European Research Council (Grant number 669306).

\section{The stationary case: E. Wong's paper revisited} 
In \cite{Wong}, example (E), the following Fokker-Planck equation was considered: 
\begin{equation*}
\partial_u^2 [(1+u^2) p] + (2\alpha-1)\partial_u (up) = \partial_t p, 
\end{equation*}
where $\alpha > 0$ and $u \in \mathbb{R}$. Using the separation of variables technique, the following singular Sturm-Liouville equation was obtained (see eq. (12) in \cite{Wong}): 
\begin{equation*}
\frac{d}{du}\left[\frac{1}{(1+u^2)^{\alpha-1/2}}\frac{d\phi}{du}(u)\right] = -\frac{\lambda}{(1+u^2)^{\alpha+1/2}} \phi(u), 
\end{equation*}
or equivalently, 
\begin{equation}\label{ED1}
(1+u^2) \frac{d^2\phi}{d^2u}+ (1-2\alpha)u\frac{d\phi}{du} = - \lambda \phi. 
\end{equation}
The diffusion operator displayed in the LHS of \eqref{ED1} is an instance of \eqref{Generator}: $\mathscr{L}_{(1/2)-\alpha, 0}$. Its discrete spectrum consists of the finite sequence of eigenvalues (\cite{RWAK}):
\begin{equation*}
\lambda_n = n(2\alpha - n), \quad 0 \leq n \leq M:=  [\alpha-1],
\end{equation*}
and the corresponding eigenfunctions are `symmetric' Routh-Romanovski polynomials: 
\begin{equation}\label{Rod}
R_n^{(\alpha)}(u) := (1+u^2)^{\alpha+(1/2)} \frac{d^n}{du^n}[(1+u^2)^{n-\alpha-(1/2)}].
\end{equation}
They are finitely-orthogonal with respect to the symmetric (Student) weight (\cite{RWAK}):
\begin{equation}\label{Cauchy}
W(u) := \frac{1}{(1+u^2)^{\alpha+1/2}},
\end{equation}
and may be expressed through Jacobi polynomials $J_n^{(-\alpha-1/2, -\alpha-1/2)}(iu)$ with imaginary variable and negative parameters as (whence the name pseudo-Jacobi polynomials): 
\begin{align*}
R_n^{(\alpha)}(u) &= c_n i^nJ_n^{(-\alpha-1/2, -\alpha-1/2)}(iu)
\\ &= c_ni^n \frac{((1/2)-\alpha)_n}{n!}{}_2F_1\left(-n, n-2\alpha, \frac{1}{2}-\alpha; \frac{1-iu}{2}\right), \quad n \leq M.
\end{align*}
Here $c_n = (-2)^n n!$ is a normalisation constant which may be computed by comparing the leading coefficients of both sides (see for e.g. \cite{RWAK}, eq. (31), for the leading coefficient of $R_n^{(\alpha)}$) and ${}_2F_1$ is the Gauss hypergeometric function  (\cite{AAR}). As to the continuous spectrum of $\mathscr{L}_{(1/2)-\alpha, 0}$, it is parametrized by $\mu \in \mathbb{R}$: 
\begin{equation*}
\lambda = \alpha^2+\mu^2
\end{equation*}
and the corresponding eigenfunctions were given in \cite{Wong} by: 
\begin{equation*}
\phi_{\alpha, \mu} := (u+\sqrt{1+u^2})^{i\mu} (1+u^2)^{1/2} {}_2F_1\left(-\alpha, \alpha+1, 1+i\mu; \frac{1}{2} + \frac{u}{2\sqrt{1+u^2}} \right). 
\end{equation*}
However, E. Wong did not provide any detail about the derivation of $\phi_{\alpha, \mu}$  and the correct expression is rather given by the following: 
\begin{proposition}
We have
\begin{align}\label{Rep1}
\phi_{\alpha, \mu}(u) &  = (u+\sqrt{1+u^2})^{i\mu}(1+u^2)^{\alpha/2} {}_2F_1\left(-\alpha, \alpha+1, 1+i\mu; \frac{1}{2} + \frac{u}{2\sqrt{1+u^2}} \right) \nonumber 
\\& =  \frac{1}{2^{i\mu}(1+u^2)^{(i\mu-\alpha)/2}}  {}_2F_1\left(\frac{1+\alpha+i\mu}{2}, \frac{i\mu-\alpha}{2}, 1+i\mu; \frac{1}{1+u^2} \right).
\end{align}
\end{proposition}
\begin{proof}
First of all, note that the quadratic transformation (see e.g. \cite{Erd}, p.112):
\begin{equation}\label{Quad1}
{}_2F_1(a,1-a,c; u)  = (1-u)^{c-1}{}_2F_1\left(\frac{c-a}{2}, \frac{c+a-1}{2}, c; 4u(1-u)\right), 
\end{equation}
transforms the expression of $\phi_{\alpha, \mu}$ given by Wong into: 
\begin{equation*}
\phi_{\alpha, \mu}(u) =  \frac{1}{2^{i\mu}(1+u^2)^{(i\mu-1)/2}}  {}_2F_1\left(\frac{1+\alpha+i\mu}{2}, \frac{i\mu-\alpha}{2}, 1+i\mu; \frac{1}{1+u^2} \right). 
\end{equation*}
As a matter of fact, it is natural  to seek an eigenfunction of $\mathscr{L}_{(1/2)-\alpha, 0}$ of the form\footnote{We discard the factor $2^{i\mu}$.}: 
\begin{equation*}
\phi_{\alpha, \mu}(u) = f\left(\frac{1}{1+u^2}\right), \quad y = \frac{1}{1+u^2},
\end{equation*}
for some smooth function $f$. In this respect, straightforward computations show that $f$ satisfies: 
\begin{equation*}
4y^2(1-y)f''(y) + \left[4(\alpha+1)y - (4\alpha+6)y^2\right] f'(y) = -\lambda f(y). 
\end{equation*}
Now, assume further that
\begin{equation*}
f(y) = y^{\beta}g(y), y \in (0,1], 
\end{equation*}
for some power $\beta \in \mathbb{R}$ and a smooth function $g$. Then 
\begin{align}\label{E2}
4y^2(1-y)g''(y) + \left[8\beta y(1-y) + y(4(\alpha+1) - (4\alpha+6)y)\right] & g'(y)  = - \left[\lambda + \right. 
\nonumber 
\\& \left. \beta(4(\alpha+1) - (4\alpha+6) y) + 4\beta(\beta-1)(1-y)\right] g(y).
\end{align}
Next we take $\lambda = \alpha^2+\mu^2$ and choose $\beta$ as a root of 
\begin{equation*}
4\beta(\beta-1) + 4(\alpha+1)\beta + \lambda = 0,
\end{equation*}
that is $\beta = (\pm i\mu - \alpha)/2$. If $\beta = (i\mu - \alpha)/2$, then 
\begin{equation*}
\beta(4\alpha+6) +4\beta(\beta-1) = (i\mu-\alpha)(\alpha+1+i\mu)
\end{equation*}
and 
\begin{equation*}
8\beta(1-y) + (4(\alpha+1) - (4\alpha+6)y) = 4(i\mu+1) - 4y \left[ i\mu + \frac{3}{2}\right].
 \end{equation*}
Keeping in mind \eqref{E2} and the hypergeometric equation (\cite{Erd}):
\begin{equation*}
y(1-y)g''(y) + (c - (a+b+1)y) g'(y) = ab g(y),
\end{equation*}
we get $f(y) = \phi_{\alpha,\mu}(y)$. If $\beta = -(i\mu + \alpha)/2$, then we similarly get 
\begin{equation*}
\beta(4\alpha+6) +4\beta(\beta-1) = (i\mu+\alpha)(i\mu - \alpha-1)
\end{equation*}
and 
\begin{equation*}
8\beta(1-y) + (4(\alpha+1) - (4\alpha+6)y) =  4(1-i\mu) + 4y \left[ \frac{3}{2}- i\mu \right],
\end{equation*}
leading to $\phi_{\alpha, -\mu} = \overline{\phi_{\alpha, \mu}}$. Note that if $\beta = (i\mu-\alpha)/2$, then the second linearly-independent solution of the hypergeometric equation which is regular around $y=0$ is given by (\cite{AAR}, \cite{Erd}):
\begin{equation*}
g(y) = y^{-i\mu}{}_2F_1\left(\frac{1+\alpha-i\mu}{2}, \frac{-i\mu-\alpha}{2}, 1-i\mu; y\right), 
\end{equation*} 
so that $f(y) = y^{(i\mu-\alpha)/2}g(y) = \phi_{\alpha, -\mu}(y) = \overline{\phi_{\alpha, \mu}(y)}$. 

\end{proof}

\subsection{The semi-group density and associated Legendre function on the cut}
By the general expansion theorem for singular Sturm-Liouville operators (\cite{Tit}), Chapter III), the semi-group density of the Hua-Pickrell process considered in \cite{Wong} admits the following expansion:  
\begin{align}\label{SDD}
q_t^{(\alpha)}(v,u) & = \frac{1}{(1+u^2)^{\alpha +(1/2)}} \left\{\sum_{n=0}^M e^{-n(2\alpha-n)t}\frac{R_n^{(\alpha)}(v)R_n^{(\alpha)}(u)}{(||R_n^{(\alpha)}||_2)^2} + \frac{1}{2\pi}
\int_{\mathbb{R}} e^{-(\alpha^2+\mu^2)t} \phi_{\alpha,-\mu}(v) \phi_{\alpha,\mu}(u) d\mu\right\},
\end{align}
where 
\begin{equation*}
(||R_n^{(\alpha)}||_2)^2:= \int_{\mathbb{R}} R_n^{(\alpha)}(u)R_n^{(\alpha)}(u) \frac{du}{(1+u^2)^{\alpha+(1/2)}}.
\end{equation*}
This squared $L^2$-norm may be computed using the Rodrigues formula \eqref{Rod} and integration by parts as (see also Corollary 1 in \cite{MMH}): 
\begin{equation}\label{SqNor}
(||R_n^{(\alpha)}||_2)^2 = \frac{\pi n!}{2^{2\alpha-2n}(\alpha-n)} \frac{\Gamma(2\alpha+1-n)}{[\Gamma(\alpha-n+(1/2))]^2},
\end{equation}
which agrees with equation (34) in \cite{Wong}. 

The semi-group density $q_t^{(\alpha)}$ may be written as a single integral of associated Legendre functions on the cut $(-1,1)$, known also as Ferrer functions (\cite{Erd}, eq. (6), p.143): 
\begin{equation}\label{Ferrer1}
P_{\alpha}^{\mu}(z) := \frac{1}{\Gamma(1-\mu)}\left(\frac{1+z}{1-z}\right)^{\mu/2} {}_2F_1(-\alpha, 1+\alpha, 1-\mu; (1-z)/2), \quad z \in (-1,1). 
\end{equation} 
This idea is actually motivated by the fact that \eqref{Rep1} readily implies:
\begin{equation}\label{Ferrer}
\phi_{\alpha, \mu}(u) = (1+u^2)^{\alpha/2}\Gamma(1+i\mu) P_{\alpha}^{-i\mu}\left(-\frac{u}{\sqrt{1+u^2}}\right) 
\end{equation}
and similarly for $\phi_{\alpha, -\mu}(v)$. However, the connection between `symmetric' Routh-Romanovski polynomials and Ferrer functions is not so obvious and we have not been able to find any reference where it is clearly stated. 
For the reader's convenience, we prove it below using variable transformations of the Gauss hypergeometric function. 

\begin{proposition}\label{LemRom}
Let $\alpha \geq 1$. Then, for any $n \leq [\alpha-1]$,
\begin{align*}
R_n^{(\alpha)}(v) = (-1)^n\frac{2^{n-\alpha}\sqrt{\pi}\Gamma(1+2\alpha-n)}{\Gamma(\alpha+1/2-n)} (1+v^2)^{\alpha/2} P_{\alpha}^{n-\alpha}\left(-\frac{v}{\sqrt{1+v^2}}\right).
\end{align*}
\end{proposition}
\begin{proof}
Appealing again to \eqref{Quad1}, the Ferrer function may be written as: 
\begin{equation*}
P_{\alpha}^{\mu}\left(-\frac{v}{\sqrt{1+v^2}}\right) = \frac{2^{\mu}}{\Gamma(1-\mu)} (1+v^2)^{\mu/2}{}_2F_1\left(\frac{1-\mu+\alpha}{2}, -\frac{\mu+\alpha}{2}, 1-\mu; \frac{1}{1+v^2}\right), \quad v \in \mathbb{R}. 
\end{equation*} 
Secondly, we apply the following quadratic transformation (\cite{Erd}, p. 112) : 
\begin{equation*}
{}_2F_1\left(\frac{a}{2}, b - \frac{a}{2}, b+\frac{1}{2}, \frac{z^2}{4(z-1)}\right) = (1-z)^{a/2}{}_2F_1(a, b, 2b; z),
\end{equation*}
with $a= -(\mu+\alpha) = -n, b = (1/2) - \mu, z = 2/(1-iv)$ 
\begin{multline*}
(1+v^2)^{\alpha/2}P_{\alpha}^{n-\alpha}\left(-\frac{v}{\sqrt{1+v^2}}\right) = \frac{2^{n-\alpha}(1+v^2)^{n/2}}{\Gamma(1+\alpha-n)} \left(\frac{1+v^2}{(v+i)^2}\right)^{-n/2} 
\\ {}_2F_1\left(\frac{1}{2}+\alpha-n, -n, 1+2\alpha-2n; \frac{2}{1-iv}\right),
\end{multline*} 
where for determinacy purposes we can assume that $v \geq 0$ so that $[(v+i)^2]^{n/2} = (v+i)^n$. This assumption is not a loss of generality since the orthogonality measure \eqref{Cauchy} is even so that $R_n^{(\alpha)}(-v) = (-1)^nR_n^{(\alpha)}(v)$ and since the same symmetry relation is satisfied by $P_{\alpha}^{n-\alpha}$ (\cite{Erd}, p.144, (14)). Finally, we appeal to the following transformation:
\begin{equation}\label{Recip}
\frac{(b)_n}{(c)_n} (-z)^{n} {}_2F_1\left(-n, 1-c-n, 1-b-n; \frac{1}{z}\right) = {}_2F_1\left(-n, b, c, z\right),
\end{equation}
which is readily checked by expanding the LHS and using the identity: 
\begin{equation*}
(1-c-n)_{n-k} = (-1)^{n-k} \frac{(c)_n}{(c)_k}, \quad 0 \leq k \leq n.
\end{equation*}
We apply it with $b = (1/2) +\alpha-n, c = 1+2\alpha-2n$ to get: 
\begin{multline*}
(1+v^2)^{\alpha/2} P_{\alpha}^{n-\alpha}\left(-\frac{v}{\sqrt{1+v^2}}\right)  =  i^n \frac{(-1)^n2^{2n-\alpha}(1/2+\alpha-n)_n}{(1+2\alpha-2n)_n\Gamma(1+\alpha-n)}  {}_2F_1\left( -n, n-2\alpha, \frac{1}{2}-\alpha; \frac{1-iv}{2}\right) 
\\ = i^n \frac{(-1)^n2^{2n-\alpha}\Gamma(1/2+\alpha)\Gamma(1+2\alpha-2n)}{\Gamma(1+2\alpha-n)\Gamma(1+\alpha-n)\Gamma(\alpha -n + 1/2)} 
\\ {}_2F_1\left( -n, n-2\alpha, \frac{1}{2}-\alpha; \frac{1-iv}{2}\right)  = \frac{(-1)^n2^{\alpha}\Gamma(1/2+\alpha)n!}{\sqrt{\pi}(1/2-\alpha)_n\Gamma(1+2\alpha-n)}i^n J_n^{(-\alpha-1/2, -\alpha-1/2)}(iv),
\end{multline*}
where the last equality follows from Legendre duplication's formula: 
\begin{equation*}
\Gamma(2z+1) = \frac{2^{2z-1}}{\sqrt{\pi}} \Gamma(z+\frac{1}{2}) \Gamma(z+1).
\end{equation*}
As a result,
\begin{align*}
R_n^{(\alpha)}(v) &= (-2^n)n! i^n J_n^{(-\alpha-1/2, -\alpha-1/2)}(iv)
\\& = \frac{2^{n-\alpha}\sqrt{\pi}(1/2-\alpha)_n\Gamma(1+2\alpha-n)}{\Gamma(\alpha+1/2)} (1+v^2)^{\alpha/2} P_{\alpha}^{n-\alpha}\left(-\frac{v}{\sqrt{1+v^2}}\right),
\end{align*}
and the lemma follows from the identity
\begin{equation*}
(1/2-\alpha)_n = (-1)^n\frac{\Gamma(\alpha+1/2)}{\Gamma(\alpha -n+1/2)}.
\end{equation*}

\end{proof}
\begin{rem}
When $\alpha \geq 1$ is an integer, the Ferrer function satisfies (\cite{Erd}, eq. (17), p. 144): 
\begin{equation*}
\Gamma(2\alpha -n +1)P_{\alpha}^{n-\alpha}\left(-\frac{v}{\sqrt{1+v^2}}\right) = (-1)^{\alpha-n} \Gamma(\alpha+1)P_{\alpha}^{\alpha-n}\left(-\frac{v}{\sqrt{1+v^2}}\right)
\end{equation*}
so that  
\begin{align*}
R_n^{(\alpha)}(v) = \frac{(-1)^{\alpha}2^{n-\alpha}\sqrt{\pi}\Gamma(\alpha+1)}{\Gamma(\alpha+1/2-n)} (1+v^2)^{\alpha/2} P_{\alpha}^{\alpha-n}\left(-\frac{v}{\sqrt{1+v^2}}\right).
\end{align*}
\end{rem}

Keeping in mind \eqref{SqNor} and using Proposition \ref{LemRom}, it follows that:
\begin{equation*}
\frac{R_n^{(\alpha)}(v)R_n^{(\alpha)}(u)}{(||R_n^{(\alpha)}||_2)^2} =  \frac{\Gamma(2\alpha+1-n)(\alpha-n)}{n!}[(1+u^2)(1+v^2)]^{\alpha/2} P_{\alpha}^{n-\alpha}\left(-\frac{u}{\sqrt{1+u^2}}\right)P_{\alpha}^{n-\alpha}\left(-\frac{v}{\sqrt{1+v^2}}\right).
\end{equation*}
Besides, the discrete and continuous spectral values may be factorized as: 
\begin{equation*}
\alpha^2+\mu^2 = (\alpha+i\mu)(\alpha-i\mu) \quad \quad n(2\alpha-n) = (\alpha + n-\alpha)(\alpha+\alpha-n). 
\end{equation*}
Consequently, 
\begin{corollary}\label{Cor1}
For any $\alpha > 0$, the semi-group density $q_t^{(\alpha)}$ may be written as: 
\begin{align*}
q_t^{(\alpha)}(v,u) & = \frac{(1+v^2)^{\alpha/2}}{(1+u^2)^{(\alpha+1)/2}}\int e^{(\alpha^2 - \mu^2)t}P_{\alpha}^{-\mu}\left(-\frac{v}{\sqrt{1+v^2}}\right)P_{\alpha}^{\mu}\left(-\frac{u}{\sqrt{1+u^2}}\right)\kappa(d\mu)
\end{align*}
where
\begin{equation*}
\kappa(d\mu) = \sum_{n=0}^{[\alpha-1]}\frac{\Gamma(2\alpha+1-n)(\alpha-n)}{n!}\delta_{\alpha - n} + \frac{\Gamma(1- \mu)\Gamma(1+\mu)}{2i\pi}{\bf 1}_{i\mathbb{R}}(\mu) d\mu. 
\end{equation*}
\end{corollary}

\begin{rem}
Since $q_t^{\alpha}(v,\cdot)$ is a probability density and since 
\begin{equation*}
\int_{\mathbb{R}} \sum_{n=0}^M e^{-n(2\alpha-n)t}\frac{R_n^{(\alpha)}(v)R_n^{(\alpha)}(u)}{(||R_n^{(\alpha)}||_2)^2}\frac{du}{(1+u^2)^{\alpha +(1/2)}} = 1, 
\end{equation*}
then we must have:
\begin{equation}\label{Integral0}
\int_{\mathbb{R}} \int_{\mathbb{R}} e^{-\mu^2t} \phi_{\alpha,-\mu}(v) \phi_{\alpha,\mu}(u) d\mu \frac{du}{(1+u^2)^{\alpha+(1/2)}} = 0,
\end{equation}
for all $v \in \mathbb{R}$. In terms of Ferrer functions, the double integral \eqref{Integral0} may be expressed as:
\begin{equation*}
(1+v^2)^{\alpha/2}\int_{\mathbb{R}} e^{-\mu^2t} P_{\alpha}^{-i\mu}\left(-\frac{v}{\sqrt{1+v^2}}\right)|\Gamma(1+i \mu)|^2d\mu  \int_{-1}^1P_{\alpha}^{i\mu}(u)(1-u^2)^{\alpha/2-1}du,
\end{equation*}
which vanishes by the virtue of Lemma 2.13 in \cite{CDD}. If $\mu=0, \alpha \in \mathbb{N} \setminus \{0\}$, then $P_{\alpha}^0$ reduces to a Legendre polynomial. In this case, \eqref{Integral0} vanishes simply due to the orthogonality of Legendre polynomials with respect to Lebesgue measure on $[-1,1]$ when $\alpha$ is even and due to their parity otherwise.
\end{rem}

\subsection{Integer values of $\alpha$: Wong's second formula}
In this case, Proposition \eqref{LemRom} allows also to prove the `peculiar' formula displayed in \cite{Wong}, equation (36), whose proof was ommitted. However, the computations are tedious and we do not detail them here since they do not lead to an easier formula for the semi-group density compared to the one obtained in Corollary \ref{Cor1}. Nonetheless, the major steps towards Wong's formula are as follows: 
\begin{itemize}
\item Use \eqref{Rep1} to expand $\phi_{\alpha, \mu}(u)$ as: 
\begin{equation*}
\phi_{\alpha,\mu}(u) = [\cosh(y)]^{\alpha} e^{i\mu y}\sum_{n=0}^{\alpha}\frac{(-\alpha)_n(1+\alpha)_n}{n!(1+i\mu)_n}\left(\frac{1+\tanh(y)}{2}\right)^n,
\end{equation*}
where we set $u:= \sinh(y)$, and do similarly for $\phi_{\alpha, \mu}(v)$ with $v := \sinh(z)$: 
 \begin{equation*}
\phi_{\alpha,-\mu}(v) = [\cosh(z)]^{\alpha} e^{-i\mu z}\sum_{m=0}^{\alpha}\frac{(-\alpha)_m(1+\alpha)_m}{m!(1-i\mu)_m}\left(\frac{1+\tanh(z)}{2}\right)^m.
\end{equation*}
\item Consider the integral: 
\begin{equation*}
\frac{1}{2\pi}\int_{\mathbb{R}} e^{-(\alpha^2+\mu^2)t}  \frac{d\mu}{(1-i\mu)_m(1+i\mu)_n} =  \frac{e^{-\alpha^2 t}}{2\pi}\int_{\mathbb{R}} e^{-\mu^2t} \frac{d\mu}{(m-i\mu)\cdots(1-i\mu)(1+i\mu)\cdots(n+i\mu)}
\end{equation*} 
and decompose 
\begin{equation*}
\frac{1}{(m-i\mu)\cdots(1-i\mu)(1+i\mu)\cdots(n+i\mu)} = \sum_{k=1}^m\frac{a_k(m,n)}{k-i\mu} + \sum_{k=1}^n\frac{b_k(m,n)}{k+i\mu}
\end{equation*}
where 
\begin{equation*}
a_k(m,n) = \frac{(-1)^{k-1}k}{(m-k)!(n+k)!}, \quad b_k(m,n) = \frac{(-1)^{k-1}k}{(n-k)!(m+k)!}.
\end{equation*}
\item Appeal to the integrals
\begin{equation*}
\frac{1}{k\pm i\mu} = \int_0^{\infty}e^{-(k\pm i\mu)v} dv, \quad \int_{\mathbb{R}} e^{-\mu^2 t}e^{iw\mu} d\mu = \sqrt{\frac{\pi}{t}}e^{-w^2/(4t)}, \, w \in \mathbb{R},
\end{equation*}
and to Fubini Theorem. Doing so gives rise to the incomplete Gamma integral
\begin{equation*}
\int_0^{\infty} e^{-kv}e^{-(v \pm(y-z))^2/(4t)} dv,
\end{equation*}
appearing in Wong's formula.
\item Reverse the summation order and arrange the different terms of the $(n,m)$-sums to get a product of two Ferrer functions $P_{\alpha}^{k-\alpha}(-\tanh(y))P_{\alpha}^{k-\alpha}(-\tanh(z))$ to which Lemma \ref{LemRom} applies. 
\end{itemize}

\section{The HP diffusion with general parameters}
In this section, we consider the HP diffusion with general parameters $A, K \in \mathbb{R}$ (not necessarily stationary) whose generator is given by $\mathscr{L}_{A,K}$. We start with a review of its spectral resolution then prove an intertwining property between 
$\mathscr{L}_{A,K}$ and $\mathscr{L}_{2-A,-K}$. In particular, this intertwining is deterministic in the sense that it does not involve a Markov kernel. Afterwards, we prove an identity which maps a `symmetric' Romanovski polynomial $R_n^{(\alpha)}$ to a `non-symmetric' one recalled below. This identity follows from the Cauchy Beta integral and may be seen as a `random' intertwining between $\mathscr{L}_{A,0}$ and $\mathscr{L}_{A,K}$. In the last section we notice that these intertwinings are closely related to some absolute continuity relations between the law of the HP diffusion with different parameters given by Girsanov's Theorem.

\subsection{Spectral resolution of the HP generator} 
The `non-symmetric' Routh-Romanovski polynomials may be defined by (there is a misprint in \cite{RWAK}, eq. (97)):
\begin{align*}
R_{n}^{(\alpha, K)}(u) & := c_ni^nJ_n^{(-\alpha-1/2-iK/2, -\alpha-1/2+iK/2)}(iu) 
\\& = c_ni^n \frac{((1-iK)/2)-\alpha)_n}{n!}{}_2F_1\left(-n, n-2\alpha, \frac{1-iK}{2}-\alpha; \frac{1-iu}{2}\right).
\end{align*}
They are eigenfunctions of $\mathscr{L}_{A,K}$ with eigenvalues $-n(2\alpha-n)$ and are finitely orthogonal with respect to the weight:
\begin{equation*}
W(u)e^{-K\cot^{-1}(u)} = \frac{1}{(1+u^2)^{\alpha+1/2}}e^{-K\cot^{-1}}(u).   
\end{equation*}
Note in passing that the eigenvalues do not depend on $K$ since the operator
\begin{equation*}
K\frac{d}{du} 
\end{equation*}
lowers the polynomial degrees. Similar to $\mathscr{L}_{A,0}$, the HP generator $\mathscr{L}_{A,K}$ admits a continuous spectrum of the form $-(\alpha^2+\mu^2)$ corresponding to (\cite{KS}, \cite{LCV}, \cite{Ner}): 
\begin{equation*}
\phi_{\alpha, K, \mu}(u) := {}_2F_1\left(-\alpha - i\mu, -\alpha + i\mu, \frac{1-iK}{2}-\alpha; \frac{1-iu}{2}\right). 
\end{equation*}
Therefore, the semi-group density of the HP diffusion admits an eigenfunction expansion similar to \eqref{SDD} which may be written as a single integral through the hypergeometric function: 
\begin{equation*}
{}_2F_1\left(-\alpha - \mu, -\alpha + \mu, \frac{1-iK}{2}-\alpha; \frac{1-iu}{2}\right), \quad \mu \in \{n-\alpha, n \leq [\alpha-1]\} \cup i\mathbb{R}.
\end{equation*} 

\subsection{Intertwining}\label{Intertwining} 
Recall the notations $A - 1/2 = -\alpha$ and recall from \cite{Mat-Yor}, (paragraph 5.1) that the unitary operator:
\begin{equation*}
f \mapsto fW, \quad W(u) = \frac{1}{(1+u^2)^{(1/2)+\alpha}} = \frac{1}{(1+u^2)^{1-A}},
\end{equation*}
intertwines $\mathscr{L}_{A,0}$ and $\mathscr{L}_{2-A, 0}$. This intertwining property extends to general parameters as follows: 
\begin{lemma}\label{Intert}
For any $A, K \in \mathbb{R}$ and any smooth function $f$, 
\begin{equation*}
\mathscr{L}_{A,K}\left(f\frac{e^{K\cot^{-1}}}{W}\right)(u) = \frac{e^{K\cot^{-1}}(u)}{W(u)}\left[\mathscr{L}_{2-A,-K}(f)(u) + 2(1-A)\right]. 
\end{equation*}
Moreover, this is the only intertwining relation of this form. 
\end{lemma}
\begin{proof}
Let $f$ be a smooth function. Then Straightforward computations show that \eqref{Generator} acts on functions of the form
\begin{equation*}
u \mapsto f(u)(1+u^2)^ae^{b\cot^{-1}(u)}, \quad a,b \in \mathbb{R},
\end{equation*}
as:
\begin{multline*}
(1+u^2)^ae^{b\cot^{-1}(u)}\left\{(1+u^2)\partial_u^2f + [(2A+4a)u+K-2b]\partial_uf + 2a f\right. \\ \left. + \frac{4a(A+a-1)u^2+2[aK+b(1-A-2a)]u +b^2-Kb}{1+u^2}f\right\}
\end{multline*}
Hence, we readily see that the couple $(a,b) = (1-A, K)$ annihilates the last term in which case $2a = 2(1-A), 2A+4a = 2(2-A), K-2b = -K$. One may also seek parameters $(a,b)$ such that 
\begin{equation*}
b^2 - Kb = 4a(A+a-1) \neq 0, \quad aK=b(A+2a-1).
\end{equation*}
This forces in particular $a \neq 0, b \neq 0$. If $K = 0$ then $a=(1-A)/2$ which lead to the contradiction:
\begin{equation*}
b^2 = -(A-1)^2,
\end{equation*}
since $b$ takes real values. Otherwise, $K \neq 0$ and in turn $A+2a-1 \neq 0$ so that 
 \begin{equation*}
b = \frac{aK}{A+2a-1}. 
\end{equation*}
Consequently, 
\begin{equation*}
b^2 - Kb = K^2 \frac{(1-A)^2-(A+2a-1)^2}{4(A+2a-1)^2} = -K^2\frac{a(A+a-1)}{(A+2a-1)^2} = 4a(A+a-1),
\end{equation*}
which has no real solution as well. 

\end{proof} 
According to this lemma, we may restrict the range of the parameters $(A,K)$ to $[1,\infty[ \times \mathbb{R}_+$ in our study of the HP diffusion. In Section \ref{exponentialfunctionals} we shall see that the intertwining of Lemma \ref{Intert} is ultimately related to the absolute continuity property between the laws of the corresponding HP processes by means of Girsanov's Theorem.

\subsection{From `symmetric' to `non-symmetric' Routh-Romanovski polynomials}
Observe that $R_{n}^{(\alpha, K)}$ is a deformation of $R_{n}^{(\alpha, 0)} = R_n^{(\alpha)}$ which adds $(-iK/2, iK/2)$ to the parameters $(1/2-\alpha, 1/2-\alpha)$ of the Jacobi polynomials. On the one hand, this deformation reminds, yet in different from, the following identity valid for $\Re(c) > 0$ (see e.g. (1.18) in \cite{Ask-Fit}):
\begin{equation}\label{Bateman}
(1-u)^{a+c} \frac{J_n^{(a+c, a-c)}(u)}{J_n^{(a+c, a-c)}(1)}= \frac{\Gamma(a+c+1)}{\Gamma(a+1)\Gamma(c)}\int_u^1(1-y)^{a} \frac{J_n^{(a,b)}(u)}{J_n^{(a,b)}(1)}(y-u)^{c-1}dy, \,\, u \in (-1,1),
 \end{equation}
which follows from Euler's first Beta integral. Remarkably, a similar fractional integral was the key ingredient used in \cite{Duf} (see eq. (4.16)) for the derivation of the density of the exponential functional of Brownian motion $\mathcal{A}_t^{(\mu)}$ (see Section \ref{exponentialfunctionals} below). On the other hand, it is tempting and quite interesting to seek an intertwining operator between HP generators corresponding to $(A,K)$ and $(A,0)$ (see equation (\ref{muzero}) in the next paragraph). In the next proposition, we prove an analogue of \eqref{Bateman} relying this time on Cauchy's Beta integral recalled below. Note that a variant of this integral was already used in the study of the Hua-Pickrell measure (\cite{Bor-Ols}). 
\begin{proposition}\label{Prop3}
For any $n \leq [\alpha-1]$ and $u \in \mathbb{R}$, 
\begin{equation*}
\frac{\Gamma(\alpha-n+1/2)}{2\pi}\lim_{\substack{w \rightarrow -iu \\ \Re(w) > 0}} \int_{\mathbb{R}} R_n^{(\alpha)}(y) \frac{dy}{(1-iy)^{\alpha+1/2}(w+iy)^{1+iK/2}} = \frac{(1+iK/2)_{\alpha-n}}{(1-iu)^{\alpha+(1+iK)/2}}R_{n}^{(\alpha, K)}(u). 
\end{equation*}
\end{proposition}
\begin{proof}
Apply the transformation \eqref{Recip} with $b= n-2\alpha, c = (1/2) - \alpha$ to write $R_n^{(\alpha)}$ as: 
\begin{align*}
R_n^{(\alpha)}(y) & = c_n i^nJ_n^{(-\alpha-1/2, -\alpha-1/2)}(iy) 
\\& = c_n(-i)^n\frac{(n-2\alpha)_n}{n!}\frac{(1-iy)^n}{2^n} {}_2F_1\left(-n, \frac{1}{2}+\alpha-n, 1+2\alpha-2n; \frac{2}{1-iy}\right)
\\& = \frac{c_ni^n\Gamma(2\alpha-n+1)}{n!\Gamma(2\alpha-2n+1)}\sum_{m=0}^n(-1)^m\binom{n}{m}\frac{(\alpha-n+1/2)_m}{(2\alpha-2n+1)_m}\frac{2^{m-n}}{(1-iy)^{m-n}}.
\\& = \frac{c_ni^n\Gamma(2\alpha-n+1)}{n!\Gamma(\alpha-n+1/2)}\sum_{m=0}^n(-1)^m\binom{n}{m}\frac{\Gamma(\alpha-n+1/2+m)}{\Gamma(2\alpha-2n+1+m)}\frac{2^{m-n}}{(1-iy)^{m-n}}.
\end{align*}
Similarly, 
\begin{align*}
R_n^{(\alpha, K)}&(u)  = c_n i^nJ_n^{(-\alpha-1/2-iK/2, -\alpha-1/2+iK/2)}(iy) 
\\& = c_n(-i)^n\frac{(n-2\alpha)_n}{n!}\frac{(1-iy)^n}{2^n} {}_2F_1\left(-n, \frac{1+iK}{2}+\alpha-n, 1+2\alpha-2n; \frac{2}{1-iy}\right)
\\& = \frac{c_ni^n\Gamma(2\alpha-n+1)}{n!\Gamma(2\alpha-2n+1)}\sum_{m=0}^n(-1)^m\binom{n}{m}\frac{(\alpha-n+(1+iK)/2)_m}{(2\alpha-2n+1)_m}\frac{2^{m-n}}{(1-iy)^{m-n}}
\\& = \frac{c_ni^n\Gamma(2\alpha-n+1)}{n!\Gamma(\alpha-n+(1+iK)/2)}\sum_{m=0}^n(-1)^m\binom{n}{m}\frac{\Gamma(\alpha-n+m+(1+iK)/2)}{\Gamma(2\alpha-2n+1+m)}\frac{2^{m-n}}{(1-iy)^{m-n}}.
\end{align*}
Next, we apply Cauchy's Beta integral (\cite{Ask}):
\begin{equation*}
\frac{1}{2\pi}\int_{\mathbb{R}} \frac{dy}{(1-iy)^a(w+iy)^b} = \frac{\Gamma(a+b-1)}{\Gamma(a)\Gamma(b)(1+w)^{a+b-1}}, \quad \Re(w) > 0, \, \Re(a+b) > 1,
\end{equation*}
with $a = \alpha -n + m+1/2, b = 1+iK/2$, to get 
\begin{multline*}
\frac{\Gamma(\alpha-n+1/2)}{2\pi} \int_{\mathbb{R}} R_n^{(\alpha)}(iy) \frac{dy}{(1-iy)^{\alpha+1/2}(w+iy)^{\alpha}} = \frac{c_ni^n\Gamma(2\alpha-n+1)}{n!\Gamma(1+iK/2)(1+w)^{\alpha+(1+iK)/2}}
\\ \sum_{m=0}^n(-1)^m\binom{n}{k}\frac{\Gamma(\alpha-n+m+(1+iK)/2)}{\Gamma(2\alpha-2n+m+1)}\frac{2^{m-n}}{(1+w)^{m-n}}. 
\end{multline*}
Letting $w \rightarrow -iu$, we are done.

\end{proof}

\subsection{Remarks on Girsanov's Theorem and exponential functionals}
\label{exponentialfunctionals}
In relation to a certain generalisation of Bougerol's identity \cite{Vak}, the authors in \cite{ADY} and \cite{AMS}  studied the diffusion whose infinitesimal generator acts on smooth functions as\footnote{The parameter $\mu$ has no relation to the previous one encoding the continuous spectrum of $\mathcal{L}_{A.K}$.}: 
\begin{equation*}
\frac{1}{2}\frac{d^2}{d^2y}+ \left[\mu \tanh(y)+\frac{\nu}{\cosh(y)}\right] \frac{d}{dy}\quad \mu,\nu\in\mathbb{R}.
\end{equation*}
With regard to \eqref{Generator1}, the corresponding diffusion is 
\begin{equation*}
\left(Y^{(A,K)}_{t/2}, t\geq0\right), \quad \mu = A - \frac{1}{2}, \nu = \frac{K}{2},
\end{equation*}
that we shall simply denote by $Y^{(\mu,\nu)}_{t}$, $t\geq0$. Recall from \cite{ADY}  that if $P^{\mu,\nu}$ denotes the probability law on the canonical space $C(\mathbb{R}_{+},\mathbb{R})$ of the diffusion $Y^{(\mu,\nu)}_{t}$, $t\geq0$, then the following absolute continuity relation holds
\begin{align}\label{Intert2}
P^{\mu,\nu}|_{\mathcal{F}^{Y}_{t}}=D^{\mu,\nu}_{t}(Y)P^{0,0}|_{\mathcal{F}^{Y}_{t}},
\end{align}
where $Y$ is the canonical coordinate process, $P^{0,0}$ denotes the Wiener measure and
\begin{align*}
D^{\mu,\nu}_{t}(Y)=\exp\left(\int_{0}^{t}\left(\mu\tanh\left(Y_{s}\right)+\frac{\nu}{\cosh(Y_{s})}\right)dY_{s}-\frac{1}{2}\int_{0}^{t}ds\left(\mu\tanh(Y_{s})+\frac{\nu}{\cosh(Y_{s})}\right)^{2}\right).
\end{align*}

We deduce from It\^{o}'s formula that an intertwining operator between HP generators corresponding to $(\mu,\nu)$ and $(\mu,0)$\,\,(equiv. $(A,K)$ and $(A,0)$) is related to the following absolute continuity relation (equation (\ref{muzero}) corrects a typo in Section 3.1 of \cite{ADY}):
\begin{align}\label{muzero}
P^{\mu,\nu}|_{\mathcal{F}^{Y}_{t}}=\exp\left(\nu\,h(Y_{t})-\nu h(Y_{0})+\nu\frac{1-2\mu}{2}\int_{0}^{t}ds\frac{\sinh(Y_{s})}{\cosh^{2}(Y_{s})}-\frac{\nu^{2}}{2}\int_{0}^{t}\frac{ds}{\cosh^{2}(Y_{s})}\right)P^{\mu,0}|_{\mathcal{F}^{Y}_{t}},
\end{align}
where 
\begin{align*}
h(x)=\int_{0}^{x}\frac{dy}{\cosh(y)}=\tan^{-1}(\sinh(x))=\frac{\pi}{2}-\cot^{-1}(\sinh(x)).
\end{align*}
Note the occurrence of the hyperbolic scarf potential \eqref{Scarf} in the RHS of \eqref{muzero}. 

The constants $\nu(1-2\mu)/2$ and $\nu^{2}/2$ in equation (\ref{muzero}) are invariant under the relation $(\mu,\nu)\leftrightarrow(1-\mu,-\nu)$ and, therefore,  for an intertwining operator between HP generators with parameters $(\mu,\nu)$ and $(1-\mu,-\nu)$\,\,(equiv. $(A,K)$ and $(2-A,-K)$)  we further obtain the simplified relation:
\begin{align}\label{Girsanov}
P^{\mu,\nu}|_{\mathcal{F}^{Y}_{t}}=e^{(1/2-\mu)t}W(\sinh(Y_{t}))\exp\left(2\nu\, h(Y_{t})\right)\, P^{1-\mu,-\nu}|_{\mathcal{F}^{Y}_{t}},
\end{align}
where $W(u)$ is given as in Section \ref{Intertwining}:
\begin{align*}
W(u)=\frac{1}{(1+u^{2})^{1/2-\mu}}=\frac{1}{(1+u^{2})^{1-A}}.
\end{align*}
One readily sees that equation (\ref{Girsanov}) is equivalent to the intertwining property of Lemma \ref{Intert} and we can restrict the range of parameters $(\mu,\nu)$ to $[1/2,\infty[\times\mathbb{R}_{+}$.

On the other hand, when $\mu<1/2$ and $\nu=0$ (equiv. $A<1$ and $K=0$), it is shown in \cite{Mat-Yor}, equation (4.2), that (\ref{Girsanov}) is also equivalent to the identity: 
\begin{align}\label{dufresneident}
e^{\mu^{2}t/2}\mathbb{E}\left[\frac{1}{\sqrt{\mathcal{A}^{(\mu)}_{t}}}\,\psi\left(\frac{1}{\mathcal{A}^{(\mu)}_{t}}\right)\right]=e^{(1-\mu)^{2}t/2}\mathbb{E}\left[\frac{1}{\sqrt{\mathcal{A}^{(1-\mu)}_{t}}}\,\psi\left(\frac{1}{\mathcal{A}^{(1-\mu)}_{t}}+2\gamma_{1/2-\mu}\right)\right],
\end{align}
for every nonnegative Borel function $\psi$ and $t>0$, where $\mathcal{A}_t^{(\mu)}$ is the exponential functional of Brownian motion given by
\begin{equation*}
\mathcal{A}_t^{(\mu)} : = \int_0^{t}e^{2B_s^{(\mu)}}ds,
\end{equation*}
and $\gamma_{1/2-\mu}$ is a Gamma variable (independent of $B$) with density: 
\begin{equation*}
u^{-(1/2+\mu)}\frac{e^{-u}}{\Gamma(1/2-\mu)}{\bf 1}_{\{u > 0\}}.
\end{equation*}
Note that the equation (\ref{dufresneident}) is reminiscent of Dufresne's identity (see e.g. \cite{Duf}); however, the additional square root factor of $1/\mathcal{A}^{(\mu)}_{t}$ is related to the intertwining $(\mu, 0) \leftrightarrow (1-\mu, 0)$ (equiv. $(A, 0) \leftrightarrow (2-A, 0)$), instead of the ``opposite drift'' case $(\mu, 0) \leftrightarrow (-\mu, 0)$ considered e.g. in \cite{Mat-Yor0}. 

We remark that it would be interesting to find a (two-dimensional) analogue of the identity (\ref{dufresneident}) corresponding to the general case $\nu\not=0$ (equiv. $K\not=0$) in equation (\ref{Girsanov}). In this regard, since for fixed $t>0$ we have  (\cite{ADY}, Proposition 1)
\begin{align*}
\sinh(Y_{t}^{(\mu,\nu)})\stackrel{d}{=}\beta_{\mathcal{A}^{(\mu)}_{t}}+\nu a^{(\mu)}_{t},\quad\text{with}\quad Y^{(\mu,\nu)}_{0}=0,\quad a^{(\mu)}_{t}:=\int_0^{t}e^{B_s^{(\mu)}}ds,
\end{align*}
where $B$ and $\beta$ are independent Brownian motions, one easily obtains from equation (\ref{Girsanov}) that:
\begin{align}\label{iden}
\mathbb{E}\left[\frac{1}{\sqrt{2\pi\mathcal{A}^{(\mu)}_{t}}}e^{-\left(u-\nu a^{(\mu)}_{t}\right)^{2}/2 \mathcal{A}^{(\mu)}_{t}}\right]=e^{(1/2-\mu)t}\eta(u)\,\mathbb{E}\left[\frac{1}{\sqrt{2\pi\mathcal{A}^{(1-\mu)}_{t}}}e^{-\left(u+\nu a^{(1-\mu)}_{t}\right)^{2}/2 \mathcal{A}^{(1-\mu)}_{t}}\right],
\end{align}
a.s. $u\in\mathbb{R}$, where $\eta(u):=W(u)e^{\nu(\pi-2\cot^{-1}(u))}$ (see also Lemma \ref{Intert}). The analysis of equation (\ref{iden}) is related to the joint distribution of the vector (see \cite{Yor1}, Section 6)
\begin{align*}
\left(\beta_{\mathcal{A}^{(\mu)}_{t}},a_{t}^{(\mu)}\right)\stackrel{d}{=}\left(\int_{0}^{t}e^{B_{s}^{(\mu)}}d\gamma_{s},a^{(\mu)}_{t}\right),\quad t>0,
\end{align*}
where $\gamma$ is another Brownian motion independent of $B$. Moreover, if $\mu<0$ (equiv. $A<1/2$), note that the HP process $\sinh(Y^{\mu,\nu}_{t})$, $t\geq0$, converges in law, as $t\to\infty$, to the so-called `subordinated perpetuity':
\begin{align*}
\int_{0}^{\infty}e^{B^{(\mu)}_{s}}d\gamma^{(\nu)}_{s}\stackrel{d}{=}\beta_{\mathcal{A}^{(\mu)}_{\infty}}+\nu a^{(\mu)}_{\infty},\quad\text{where}\quad \gamma^{(\nu)}_{s}=\gamma_{s}+\nu s,
\end{align*} 
and whose (unnormalized) probability density is given by the function $\eta(u)$ defined above (see \cite{Mat-Yor0}, Section 4).

\section{Integral representation of the HP semi-group density}
\label{IntrepHP}
The left-hand side of Equation (\ref{iden}) gives a representation formula for the HP semi-group density; however, this is not very tractable. In this section we prove an explicit integral representation through an extension of the Maass semi-group density to imaginary values of the magnetic field.

With the same notation as in Section \ref{exponentialfunctionals}, the following Fourier transform of the HP process was derived in \cite{AMS}, Proposition 6.1. If $Y^{(\mu,\nu)}_{0}=v\in\mathbb{R}$:
\begin{multline}\label{FT}
\mathbb{E}_{v}\left[\exp\left(i\lambda \sinh\left(Y^{(\mu,\nu)}_{t}\right)\right)\right] = e^{-\mu^{2}t/2} \int_{0}^{\infty} dz \int_{-\infty}^{\infty}dy\exp(i\lambda \sinh(v)e^{y}+\mu y+i\lambda\nu z)\\
\\ \frac{\lambda}{4\sinh(\lambda z/2)}\exp(-\lambda(1+e^{y})\coth(\lambda z/2))\theta_{\phi(z,y; \lambda)}\left(\frac{t}{4}\right),
\end{multline}
where 
\begin{align}\label{Hartman}
\theta_{r}(t) :=\frac{r}{(2\pi^{3}t)^{1/2}}e^{\pi^{2}/2t}\int_{0}^{\infty}e^{-w^{2}/(2t)}e^{-r\cosh(w)}\sinh(w)\sin\left(\frac{\pi w}{t}\right)dw,
\end{align}
and
\begin{align*}
\phi(z,y; \lambda) := \frac{2\lambda e^{y/2}}{\sinh(\lambda z/2)}.
\end{align*} 
Note that the double integral in (\ref{FT}) is absolutely convergent for all $\mu$ since $\theta_r(t) = O(r^N)$ as $r$ gets close to zero, for any $N > 0$ (see \cite{Mat-Yor}, p. 188 and paragraph 6.2 there). 

In order to state our result, recall from \cite{AMS} (see also \cite{Ike-Mat}) the Maass Laplacian:
\begin{equation*}
H_k := -\frac{y^2}{2}(\partial_w^2+\partial_y^2) + iky\partial_w + \frac{k^2}{2},
\end{equation*}
where $k \in \mathbb{R}$ is interpreted in the physics realm as the intensity of a constant magnetic field. This is a densely-defined and essentially-self-adjoint operator in 
\begin{equation*}
L^2\left(\mathbb{H}, dw \frac{dy}{y^2}\right),
\end{equation*}
where $\mathbb{H}$ is the Poincar\'e upper-half-plane, and reduces to the Laplace-Beltrami operator in $\mathbb{H}$ when $k = 0$: 
\begin{equation*}
H_0 = -\frac{y^2}{2}(\partial_w^2+\partial_y^2).
\end{equation*}
The semi-group density of its self-adjoint closure with initial point $z=i$ admits the following expression (\cite{AMS}, \cite{Ike-Mat}): 
\begin{multline}\label{MaassSD}
Q_{t,k}^H(i,w+iy) = \left(\frac{w+i(y+1)}{i(y+1)-w}\right)^k\frac{\sqrt{2}e^{-t/8-k^2t/2}}{(2\pi t)^{3/2}} \\ \int_{r}^{\infty}dz\frac{ze^{-z^2/(2t)}}{\sqrt{\cosh(z) - \cosh(r)}} \cosh\left\{2k\cosh^{-1}\left(\frac{\cosh(z/2)}{\cosh(r/2)}\right)\right\}.
\end{multline}
with respect to the hyperbolic volume measure $dw dy/y^2$. Here the principal determination of the power is taken and $r = d_{\mathbb{H}}(i,w+iy)$ is the hyperbolic distance: 
\begin{equation*}
\cosh(r) = 1+ \frac{w^2+(y-1)^2}{2y}.
\end{equation*}
With these notations, we have the following:

\begin{theorem}\label{maintheorem}
The semi-group density of the HP process $\sinh\left(Y^{(\mu,\nu)}_{t}\right)$, $t\geq0$, admits the following integral representation: 
\begin{equation*}
g^{(\mu,\nu)}_{t}(\sinh(v),w)= e^{(1-4\mu^{2})t/8}e^{-\nu^2t/2} \int_{0}^{\infty} dy\,y^{\mu-3/2} Q_{t,i\nu}^H(i,w-\sinh(v)y+iy), \quad w, v \in \mathbb{R}, 
\end{equation*}
where $Q_{t,i\nu}^H$ is the extension of the RHS of \eqref{MaassSD} to purely imaginary values $k = i\nu$. 
\end{theorem}

\begin{proof}
Perform the variable change $z \mapsto 2z/\lambda$ in the double integral displayed in \eqref{FT} and change the order of integration there to write: 
\begin{multline}\label{Complex}
\mathbb{E}_{v}\left[\exp(i\lambda \sinh(Y^{(\mu,\nu)}_{t}))\right] = e^{-\mu^{2}t/2} \int_{-\infty}^{\infty} dy \exp(i\lambda \sinh(v)e^{y}+\mu y) \\
\int_{0}^{\infty} dz \frac{e^{2i\nu z}}{2\sinh(z)}\exp(-\lambda(1+e^{y})\coth(z))\theta_{\phi(2z/\lambda,y; \lambda)}\left(\frac{t}{4}\right).
\end{multline}
On the other hand, Proposition 4.1 and equation (5.2) in \cite{AMS} together with Fourier inversion Theorem show that for any $\lambda > 0, k \in \mathbb{R}$,
\begin{align}\label{Mor-Maa}
\int_{0}^{\infty} dz \frac{e^{2kz}}{2\sinh(z)}\exp(-\lambda(1+e^{y})\coth(z))\theta_{\phi(2z/\lambda,y; \lambda)}\left(\frac{t}{4}\right) &= q_{t,\lambda}^M(0,y) \nonumber 
\\& = e^{t/8+k^2t/2}e^{-y/2}\int_{\mathbb{R}}dw e^{i\lambda w} Q_{t,k}^H(i,w+ie^y),
\end{align}
where $q_{t,\lambda}^M(0,y)$ is the semi-group density of the Schr\"odinger operator with Morse potential: 
\begin{equation*}
H_{\lambda}^M := -\frac{1}{2}\frac{d^2}{du^2} + \frac{\lambda^2}{2}e^{2u},
\end{equation*}
starting at the origin. Now, the LHS of \eqref{Complex} is analytic in $k$ since the integral converges absolutely for any $k \in \mathbb{R}$. Besides, we can rewrite $Q_{t,k}^H(i,w+iy)$ as
\begin{multline*}
Q_{t,k}^H(i,w+iy) = \frac{\sqrt{2}e^{-t/8-k^2t/2}}{(2\pi t)^{3/2}}e^{ik(2\arg(w+(y+1)i) - \pi)} \\ \int_{r}^{\infty}dz\frac{ze^{-z^2/(2t)}}{\sqrt{\cosh(z) - \cosh(r)}} \cosh\left\{2k\cosh^{-1}\left(\frac{\cosh(z/2)}{\cosh(r/2)}\right)\right\},
\end{multline*}
which extends analytically to $k \in \mathbb{C}$. Since the integral displayed in the RHS of \eqref{Mor-Maa} still converges absolutely for $k \in \mathbb{C}$, then \eqref{Mor-Maa} extends analytically to the whole complex plane. 
In particular, for $k = i\nu$, the inner integral in \eqref{Complex} is expressed as:
\begin{multline*}
\int_{0}^{\infty} dz \frac{e^{2i\nu z}}{2\sinh(z)}\exp(-\lambda(1+e^{y})\coth(z))\theta_{\phi(2z/\lambda,y; \lambda)}\left(\frac{t}{4}\right) =  e^{t/8+k^2t/2}e^{-y/2}\int_{\mathbb{R}}dw e^{i\lambda w} Q_{t,i\nu}^H(i,w+ie^y)
\end{multline*}
whence 
\begin{multline*}
\mathbb{E}_{v}\left[\exp(i\lambda \sinh(Y^{(\mu,\nu)}_{t}))\right] = e^{(1-4\mu^{2})t/8}e^{-\nu^2t/2}  \int_{0}^{\infty} dy e^{i\lambda \sinh(v)y}y^{\mu-3/2} \int_{\mathbb{R}}dw e^{i\lambda w} Q_{t,i\nu}^H(i,w+iy).
\end{multline*}
The double integral in the right-hand side converges absolutely for any $\mu \in \mathbb{R}$. Indeed, 
\begin{equation}\label{bound}
|Q_{t,i\nu}^H(i,w+iy)| \leq e^{\nu^2t/2} e^{2\nu \pi} Q_{t,0}^H(i,w+iy)
\end{equation}
and $Q_{t,0}^H(i,w+iy)$ is the semi-group density of the hyperbolic Brownian motion $(N_t, O_t)_{t \geq 0} \in \mathbb{H}$ solution of the following SDE: 
\begin{eqnarray*}
O_t &= & e^{W^1_t - t/2},  \\ 
N_t & = & \int_0^t O_s dW^2_s,
\end{eqnarray*}
for a planar (Euclidean) Brownian motion $(W^1_t, W^2_t)_{t \geq 0}$. Thus, 
\begin{equation*}
\int_{0}^{\infty} \frac{dy}{y^2} y^{\mu+1/2} \int_{\mathbb{R}}dw |Q_{t,i\nu}^H(i,w+iy)| \leq  e^{\nu^2t/2} e^{2\nu \pi} \mathbb{E}(O_t^{\mu+1/2}) < \infty. 
\end{equation*}
As a matter, we can invert order of integration and get the semi-group density of the HP diffusion $\sinh\left(Y^{(\mu,\nu)}_{t/2}\right)$, $t\geq0$. The theorem is proved.  

\end{proof}

\begin{rem}
From Remark 3.1. in \cite{Ike-Mat}, it follows that for any $k \in \mathbb{R}$:  \footnote{The eigenvalue is $\lambda = [-s(s-1)+B^2]/2$ and not $[-s(s-1)/2+B^2]/2$.}
\begin{equation*}
\int_{\mathbb{H}} y^s Q_{t,k}^H(i,w+iy) \frac{dw dy}{y^2} = e^{[s(s-1) - k^2]t/2}.
\end{equation*}
By the same argument used in the previous proof, this equality extends to all complex values of $k$ and shows in particular, with $s=\mu+1/2$ and $k = i\nu$, that 
\begin{equation*}
w \mapsto e^{(1-4\mu^{2})t/8} e^{-\nu^2t/2} \int_{0}^{\infty} dy y^{\mu-3/2} Q_{t,i\nu}^H(i,w-\sinh(v)y+iy),
\end{equation*}
is indeed a probability density. On the other hand, the intertwining relation proved in Lemma \ref{Intert} yields another integral representations: 
\begin{multline*}
g^{(\mu,\nu)}_{t}(\sinh(v),w) = e^{(1-4\mu^{2})t/8}e^{-\nu^2t/2} \frac{e^{2\nu \cot^{-1}(\sinh(v))}(1+(\sinh^2(v))^{1/2-\mu}}{e^{2\nu \cot^{-1}(w)}(1+w^2)^{1/2-\mu}} \\ \int_{0}^{\infty} dy\,y^{-\mu-1/2} Q_{t,-i\nu}^H(i,w-\sinh(v)y+iy).
\end{multline*}
In particular, both representations coincide for $\mu=1/2, \nu = 0$.  
\end{rem}



\section{The Hua-Pickrell particle system}

In this final section we give some interest in the multivariate analogue of the HP process: the HP particle system. More precisely, let $s$ be a complex parameter such that $\Re(s) > -1/2$, the HP particle system $X= (X_n)_{n=1}^N$ was introduced in \cite{Assio} as the $N$-dimensional diffusion satisfying the stochastic differential system:  
\begin{multline*}
dX_n(t) = \sqrt{2(1+(X_n(t))^2} dW_n(t)+ 2\left\{[(1-N-\Re(s)]X_n(t) + \Im(s)] + \sum_{j \neq n}\frac{1+(X_n(t))^2}{X_n(t)- X_j(t)}\right\} dt,  \\ 1 \leq n \leq N,
\end{multline*}
where $(W_n)_{n=1}^N$ is a $N$-dimensional Brownian motion. It was also shown there (see Lemma 4.2) that the particles never collide almost surely  and that the multivariate HP measure 
$\tau_{HP}^{(s,N)}$: 
\begin{equation*}
\tau_{HP}^{(s,N)}(dx_1, \dots, dx_N) := C_{s,N} [V(x)]^2 \prod_{n=1}^N(1+x_n^2)^{-\Re(s) - N}e^{2\Im(s)\arg(1+ix_n)}dx_n, \quad x = (x_n)_{n=1}^N \in \mathbb{R}^N,
\end{equation*}
where $C_{s,N}$ is a normalizing constant and 
\begin{equation*}
V(x) = \prod_{1 \leq n < j \leq N}(x_n-x_j)
\end{equation*}
is the Vandermonde polynomial, is the unique invariant measure for their semi-group. Note that the trigonometric identity: 
\begin{equation*}
\arg(1+iu) = \tan^{-1}u = \frac{\pi}{2} - \cot^{-1}(u), \quad u \in \mathbb{R},
\end{equation*}
allows to write 
\begin{equation*}
\tau_{HP}^{(s,N)}(dx_1, \dots, dx_N) := C_{s,N}e^{N\pi\Im(s)} [V(x)]^2 \prod_{n=1}^N(1+x_n^2)^{-\Re(s) - N}e^{-2\Im(s)\cot^{-1}(x_n)}dx_n,
\end{equation*}
in accordance with our previous computations in the one-dimensional setting. Moreover, Lemma 4.1 in \cite{Assio} asserts that $X$ is, up to a deterministic time change, the Doob transform of $N$ independent copies of the HP diffusion with generator $\mathcal{L}_{A,K}$ where:
\begin{equation*}
A  = 1-N-\Re(s) < 1/2, \quad K = 2\Im(s).
\end{equation*}  
Recalling the relations: 
\begin{equation*}
\mu = A-\frac{1}{2}, \quad \nu = \frac{K}{2},
\end{equation*}
$X$ is the Doob transform of $N$ independent copies of $\sinh\left(Y^{((1/2)-N-\Re(s), \Im(s))}\right)$. 
Consequently, the Karlin-McGregor formula allows to write the semi-group density $G_t^{(s,N)}$ of $X$ as: 
\begin{equation}\label{KM}
G_t^{(s,N)}(x,y) = e^{-\lambda_{s,N}t}\frac{V(y)}{V(x)} \det\left(g_t^{((1/2)-N-\Re(s), \Im(s))}\left[\sinh^{-1}(x_n),\sinh^{-1}(y_j)\right]\right)_{n,j=1}^N, 
\end{equation}
where $x_1 > \dots > x_n, y_1 > \dots > y_n$, and  
\begin{equation*}
\lambda_{s,N} = \frac{N(N-1)(1-2N-3\Re(s))}{3}. 
\end{equation*}
When $s = \alpha -(1/2), \alpha > 0,$ is real, we can also appeal to the semi-group density $q_t^{(N+s -1/2)}$ defined in \eqref{SDD}. In this case, Corollary \ref{Cor1} together with Adreief's identity (\cite{Dei-Gio}, p.37) yield the following result:  
\begin{proposition}
For any real $s > -1/2$, the semi-group density of the HP-particle system is given by: 
\begin{align*}
G_t^{(s,N)}(x,y) & = e^{[(N+s-1/2)^2-\lambda_{s,N}]t}\prod_{n=1}^n\frac{(1+v^2)^{(N+s-(1/2))/2}}{(1+u^2)^{(N+s+(1/2))/2}} V^2(y)
\int_{\mu_1>\dots > \mu_N}e^{(-\mu_1^2+\dots+\mu_N^2)t}  \\& \frac{\det\left(P_{N+s-1/2}^{-\mu_j}\left(-\displaystyle \frac{x_n}{\sqrt{1+x_n^2}}\right)\right)_{j,n=1}^N}{V(x)}
\frac{\det\left(P_{N+s-1/2}^{-\mu_j}\left(-\displaystyle \frac{y_n}{\sqrt{1+y_n^2}}\right)\right)_{j,n=1}^N}{V(y)}\prod_{n=1}^N\kappa(d\mu_n). 
\end{align*}
\end{proposition}

\begin{proof}
It suffices to write
\begin{align*}
G_t^{(s,N)}(x,y) & = e^{-\lambda_{s,N}t}\prod_{n=1}^n\frac{(1+v^2)^{(N+s-(1/2))/2}}{(1+u^2)^{(N+s+(1/2))/2}} \frac{V(y)}{V(x)} \det\left(q_t^{(N+s -1/2)}(x_n,y_j\right)_{n,j=1}^N
\\& =  e^{-\lambda_{s,N}t}\prod_{n=1}^n\frac{(1+v^2)^{(N+s-(1/2))/2}}{(1+u^2)^{(N+s+(1/2))/2}}\frac{V(y)}{V(x)} 
\\& \det\left(\int e^{[(N+s-1/2)^2 - \mu^2]t}P_{N+s-1/2}^{-\mu}\left(-\frac{x_n}{\sqrt{1+x_n^2}}\right)P_{N+s-1/2}^{\mu}\left(-\frac{y_j}{\sqrt{1+y_j^2}}\right)\kappa(d\mu)\right)_{n,j=1}^N. 
\end{align*}

\end{proof}
\begin{rem}
The function 
\begin{equation*}
x \mapsto \frac{\det\left(P_{N+s-1/2}^{-\mu_j}\left(-\displaystyle \frac{x_n}{\sqrt{1+x_n^2}}\right)\right)_{j,n=1}^N}{V(x)}
\end{equation*}
is a symmetric multivariable extension of the associated Legendre function on the cut. 
\end{rem}

\end{document}